\documentclass{amsart}
\usepackage{amscd,amssymb,amsopn,amsmath,amsthm,graphics,amsfonts,accents,enumerate,verbatim,calc}
\usepackage[dvips]{graphicx}
\usepackage[colorlinks=true,linkcolor=red,citecolor=blue]{hyperref}
\usepackage[all]{xy}
\usepackage{cite}
\usepackage{tikz}
\usepackage{tikz-cd}

\calclayout

\newcommand{\rt}{\longrightarrow}
\newcommand{\lrt}{\longrightarrow}

\newcommand{\st}{\stackrel}

\newcommand{\La}{\Lambda}
\newcommand{\Ga}{\Gamma}
\newcommand{\Om}{\Omega}

\newcommand{\D}{\mathbb{D} }

\newcommand{\N}{\mathbb{N} }
\newcommand{\Z}{\mathbb{Z} }

\newcommand{\CA}{\mathcal{A} }
\newcommand{\CC}{\mathcal{C} }

\newcommand{\CI}{\mathcal{I} }

\newcommand{\CM}{\mathcal{M} }

\newcommand{\CP}{\mathcal{P} }

\newcommand{\CS}{\mathcal{S} }

\newcommand{\CX}{\mathcal{X} }

\newcommand{\mmod}{{\rm{{mod\mbox{-}}}}}

\newcommand{\prj}{{\rm{prj}\mbox{-}}}

\newcommand{\Gprj}{{\Gp\mbox{-}}}

\newcommand{\Ginj}{{\Gi \mbox{-}}}

\newcommand{\op}{{\rm{op}}}

\newcommand{\add}{{\rm{add}\mbox{-}}}

\newcommand{\pd}{{\rm{pd}}}
\newcommand{\id}{{\rm{id}}}

\newcommand{\gldim}{{\rm{gl.dim}}}

\newcommand{\domdim}{{\rm{dom.dim}}}

\newcommand{\nGprj}{{n\Z\mbox{-}\Gprj}}
\newcommand{\nGinj}{{n\Z\mbox{-}\Ginj}}

\newcommand{\Gp}{{\rm{Gprj}}}

\newcommand{\Gi}{{\rm{Ginj}}}

\newcommand{\Coker}{{\rm{Coker}}}
\newcommand{\Ker}{{\rm{Ker}}}

\newcommand{\Hom}{{\rm{Hom}}}
\newcommand{\Ext}{{\rm{Ext}}}

\theoremstyle{plain}
\newtheorem{theorem}{Theorem}[section]
\newtheorem{corollary}[theorem]{Corollary}
\newtheorem{lemma}[theorem]{Lemma}

\newtheorem{proposition}[theorem]{Proposition}

\theoremstyle{definition}
\newtheorem{definition}[theorem]{Definition}

\newtheorem{remark}[theorem]{Remark}

\theoremstyle{plain}

\theoremstyle{definition}

\numberwithin{equation}{section}

\begin{document}

\title[$n\Z$-Gorenstein cluster tilting subcategories]{$n\Z$-Gorenstein cluster tilting subcategories}

\author[Javad Asadollahi, Rasool Hafezi and Somayeh Sadeghi]{Javad Asadollahi, Rasool Hafezi and Somayeh Sadeghi}

\address{Department of Mathematics, University of Isfahan, P.O.Box: 81746-73441, Isfahan, Iran}
\email{asadollahi@ipm.ir, asadollahi@sci.ui.ac.ir,}

\address{School of Mathematics, Institute for Research in Fundamental Sciences (IPM), P.O.Box: 19395-5746, Tehran, Iran}
\email{hafezi@ipm.ir}

\address{Department of Mathematics, University of Isfahan, P.O.Box: 81746-73441, Isfahan, Iran}
\email{so.sadeghi@sci.ui.ac.ir }



%\thanks{This research was in part supported by a grant from IPM (No. 93130216)}

\begin{abstract}
Let $\La$ be an artin algebra. In this paper, the notion of $n\Z$-Gorenstein cluster tilting subcategories will be introduced. It is shown that every $n\Z$-cluster tilting subcategory of $\mmod\La$ is $n\Z$-Gorenstein if and only if $\La$ is an Iwanaga-Gorenstein algebra. Moreover, it will be shown that an $n\Z$-Gorenstein cluster tilting subcategory of $\mmod\La$ is an $n\Z$-cluster tilting subcategory of the exact category $\Gprj\La$, the subcategory of all Gorenstein projective objects of $\mmod\La$. Some basic properties of $n\Z$-Gorenstein cluster tilting subcategories will be studied. In particular, we show that they are $n$-resolving, a higher version of resolving subcategories.
\end{abstract}

\maketitle
\section{Introduction}
Throughout $R$ is a fixed commutative artinian ring. An artin $R$-algebra $\Ga$ is called of finite representation type if the set of iso-classes of finitely generated indecomposable $\Ga$-modules is finite. $\Ga$ is called an Auslander algebra if it satisfies the following homological conditions
\[\gldim \Ga \leq 2 \leq \domdim \Ga,\]
where $\domdim \Ga$ denotes the dominant dimension of $\Ga$ introduced by Tachikawa \cite{Ta}.

One of the important results in the study of algebras of finite representation type, is due to Auslander \cite{Au}, proving that here is a bijective correspondence between Morita equivalence classes of Artin algebras $\La$ of finite representation type and Morita equivalence classes of Auslander algebras.\\

A higher version of Auslander's correspondence and Auslander-Reiten theory for artin algebras and related rings is developed by Iyama in a series of papers, see e.g. \cite{I}, \cite{I1}. The new notion that made Iyama's theory fundamental in representation theory, is the notion of cluster tilting modules and cluster tilting subcategories.

Let $\La$ be an artin algebra and $\mmod\La$ denote the category of finitely generated right $\La$-modules.  Let $n$ be a positive integer. A full subcategory $\CC$ of $\mmod\La$ is called an $n$-cluster tilting subcategory if it is functorially finite and $\CC=\CC^{\perp_n}={}^{\perp_n}\CC$, where
\[ \CC^{\perp_n}:= \{X \in \mmod\La \mid \Ext^i_{\La}(\CC, X)=0, \ \text{ for all} \ 0 < i<  n \}, \]
\[ {}^{\perp_n}\CC:= \{X \in \mmod\La \mid \Ext^i_{\La}(X, \CC)=0, \ \text{ for all} \ 0 < i <  n \}. \]
$M \in \mmod\La$ is called a cluster tilting module if $\add M$ is a cluster tilting subcategory.

Although cluster tilting subcategories are not abelian, they have a nice structure, known as $n$-abelian structure. The notion of $n$-abelian categories introduced and studied by Jasso \cite{Ja}. These are categories inhabited by certain exact sequences with $n+2$ terms, called $n$-exact sequences. It follows that under sufficiently general circumstances, $n$-cluster tilting subcategories of abelian categories correspond bijectively with $n$-abelian categories \cite[Theorem 3.16]{Ja}. Homological properties of these subcategories have been studied by several authors, usually under the name of `higher homological algebras'.

So, comparing to the classical case, a natural attempt that would be of interest, is to develop a version of Gorenstein homological algebra in this new context.  This is the motivation for this work.

To introduce a Gorenstein version of cluster tilting subcategories, we concentrate on certain classes of cluster tilting subcategories, introduced recently by Iyama and Jasso \cite[Definition-Proposition 2.15]{IJ}, called $n\Z$-cluster tilting subcategories. Roughly speaking these are subcategories that are closed under $n$-syzygies and $n$-cosyzygies. We then introduce $n\Z$-Gorenstein projective objects as a higher analogue of Gorenstein projective objects in $\mmod\La$, see Definition \ref{DefnZGP}.

Let $\nGprj\CC$  denote the full subcategory of all $n\Z$-Gorenstein projective modules in $n\Z$-cluster tilting subcategory $\CC.$ As an immediate consequence of the definition, it follows that the category $\nGprj\CC$ is closed under finite direct sums. Moreover, there are inclusions of categories $\prj\La \subseteq \nGprj\CC \subseteq  \CC \cap \Gprj\La$, where $\prj\La$ and $\Gprj\La$ denote the full subcategories of finitely generated projective and finitely generated Gorenstein projective modules in $\mmod\La$, respectively. It will be shown that if $\La$ is an Iwanaga-Gorenstein algebra, the latter inclusion becomes equality.
Dually, the notion of $n\Z$-Gorenstein injective modules will be defined and will be denoted by $\nGinj\CC$.

We use these two classes to define $n\Z$-Gorenstein cluster tilting subcategories, see Definition \ref{DefnZGCTS}. We show that every $n\Z$-cluster tilting subcategory $\CC$ of $\mmod\La$ is $n\Z$-Gorenstein cluster tilting subcategory if and only if $\La$ is an Iwanaga-Gorenstein algebra.

Considering $\Gprj\La$ as an exact category, it will be shown that $\nGprj\CC$ is an $n\Z$-cluster tilting subcategory of $\Gprj\La$. Moreover, we prove the following theorem.\\

\noindent {\bf Theorem 1.}\label{NClusterTilting} Let $\La$ be an Iwanaga-Gorenstein algebra and $\CC$ be an $n\Z$-cluster tilting subcategory of $\mmod\La.$
Then we have
\begin{itemize}
\item [$(1)$] The subcategory  $\nGprj\CC$ is an $n\Z$-cluster tilting subcategory in the exact category $\Gprj\La.$
\item [$(2)$] The subcategory $\underline{\nGprj\CC}$ is an $n\Z$-cluster tilting subcategory in the triangulated category $\underline{\Gprj\La}.$		
\end{itemize}

As an immediate consequence of this theorem and having in mind that $\nGprj\CC$ and $\underline{\nGprj\CC}$ are stable under cosyzygies and the $n$-th power of the suspension functor of the triangulated category $\underline{\Gprj\La}$, respectively, it will follow that $\nGprj\CC$ is a Frobenius $n$-exact category and $\underline{\nGprj\CC}$ is an $(n+2)$-angulated category.

Then we introduce the concept of $n$-resolving subcategories, as a higher analogue of resolving subcategories in an abelian category. We show that when $\La$ is  an Iwagana-Gorenstein algebra then the subcategory $\nGprj\CC$ of $n\Z$-cluster tilting subcategory $\CC$  is an  $n$-resolving subcategory of $\CC$.

Finally, we finish the paper with an observation which can be thought of as a higher version of an equivalence proved by Buchweitz, Happel and Rickard, see e.g. \cite{Bu}, \cite{Hap} and \cite{Ri}. They proved that $\underline{\Gprj\La}$ and  the quotient category $\D^b(\La ) / \D^{\text{perf}} (\La)$ are triangulated equivalent, whenever $\La$ is an Iwanaga-Gorenstein algebra. Here $\D^b (\La)$ denotes the bounded derived category of finitely generated $\La$-modules. Furthermore, $\D^{\text{perf}} (\La)$ denotes the thick subcategory of $\D^b (\La)$ consisting of objects isomorphic to bounded complexes of finitely generated projective $\La$-modules. We show that if $\La$ is an Iwanaga-Gorenstein algebra, then  the subcategory $\underline{\nGprj\CC}$ of $\underline{\Gprj\La}$ is equivalent to the subcategory
\[\Upsilon=\{P^{\bullet} \in \D^b(\La) / \D^{\text{perf}} (\La)\mid P^{\bullet}\simeq C[ni] \ \text{for some} \ C \in \CC \ \text{and} \ i \in \Z \},\]
as (n+2)-angulated categories.\\

We end the introduction by mentioning that existence of cluster tilting objects or subcategories in an abelian category is not so frequent. In practice, it is more easy and more frequent, and in some cases more important, to find cluster tilting objects or subcategories in exact categories. For instance it is of interest to study cluster-tilting objects or subcategories of Gorenstein-projective objects, see \cite{K}. In this case the theory is richer since it is supported by the well-developed theory of cluster-tilting objects or subcategories of the stable category of the Gorenstein-projectives modulo projectives.\\

\noindent {\bf Conventions.} Throughout $R$ is a commutative artinian ring, $\La$ is an artin $R$-algebra,  $\mmod\La$ denotes the category of finitely generated (right) $\La$-modules, $\CC$ is a full subcategory of $\mmod\La$ and $n$ is a positive integer.  For $M \in \mmod\La$, $\add M$ denotes the full subcategory of $\mmod\La$ consisting of all direct summands of finite direct sums of copies of $M$.

\section{Cluster tilting subcategories}\label{section2}
$n$-cluster tilting modules and subcategories play the central role in the theory of higher homological algebra. They can be considered as a higher analog of the module category. Let us recall the definition, see \cite{I} and \cite{IJ}.

\begin{definition}
A subcategory $\CC$ of $\mmod\La$ is called an $n$-cluster tilting subcategory if it is functorially finite and $\CC=\CC^{\perp_n}={}^{\perp_n}\CC$, where
\[ \CC^{\perp_n}:= \{X \in \mmod\La \mid \Ext^i_{\La}(\CC, X)=0 \ \text{ for all} \ 0 < i<  n \}, \]
\[ {}^{\perp_n}\CC:= \{X \in \mmod\La \mid \Ext^i_{\La}(X, \CC)=0 \ \text{ for all} \ 0 < i <  n \}. \]
\end{definition}

It is known that for an arbitrary module $M \in \mmod\La$, $\rm{add}\mbox{-}M$ is always functorially finite. So $\add M$ is an $n$-cluster tilting subcategory if and only if $\add M=(\add M)^{\perp_n}={}^{\perp_n}(\add M)$. In this case, $M$ is called an $n$-cluster tilting module.

Similarly, Jasso in \cite{Ja} defined the notion of an $n$-cluster tilting subcategory of an exact category. We refer the reader to the Definition 4.13 of \cite{Ja} for details.

According to \cite[Page 343]{IJ}, a subclass of $n$-cluster-tilting subcategories, called $n\Z$-cluster tilting subcategories, ``are better behaved from the viewpoint of higher homological algebra''. Throughout the paper we shall concentrate on this special class. Let us recall the definition.

We say that an $n$-abelian category $\CC$ has $n$-syzygies if for every $M \in \CC$ there exists an $n$-exact sequence
$$0\rt  L\rt   P^{n-1} \rt \cdots \rt  P^0 \rt M\rt 0	$$
in $\CC$, where $P^i$ is a projective object,  for $i \in \{0, 1,  \cdots, n-1\}$.  By abuse of notation, $L$ is called the $n$-syzygy of $M$ and denoted by $\Om^nM$. Dually, the notion of $n$-cosyzygies and $\CC$ having $n$-cosyzygies are defined. The $n$-cosyzygy of $M$ is denoted by $\Om^{-n}M$, see Definition 2.22 of \cite{IJ}.

\begin{theorem}(See \cite[Definition-Proposition  2.15]{IJ})
Let $\CC$ be an $n$-cluster tilting subcategory of $\mmod\La$. The following conditions are equivalent.
\begin{itemize}
\item [$(a)$] $\Ext_{\La}^i(\CC, \CC)= 0$, for all $i \notin n\Z.$
\item [$(b)$] $\Om^n(\CC) \subseteq \CC.$
\item [$(c)$] $\Om^{-n}(\CC) \subset \CC.$
\item [$(d)$] For each $X \in \CC$ and for each
\[ 0\rt   L\rt  M^1\rt \cdots \rt  M^n\rt N\rt  0	\]
exact sequence in $\mmod \La$ whose terms lie in $\CC$ there is an exact sequence
\[\begin{tikzcd}[column sep=tiny, row sep=tiny]
0\rar&\Hom_{\La}(X,L)\rar&\Hom_{\La}(X,M^1)\rar&\cdots\rar&\Hom_{\La}(X,M^n)\rar&\Hom_{\La}(X,N)\rar&{}\\     {}\rar&\Ext_{\La}^n(X,L)\rar&\Ext_{\La}^n(X,M^1)\rar&\cdots\rar&\Ext_{\La}^n(X,M^n)\rar&\Ext_{\La}^n(X,N)\rar&{}\\
{}\rar&\Ext_{\La}^{2n}(X,L)\rar&\Ext_{\La}^{2n}(X,M^1)\rar&\cdots\rar&\Ext_{\La}^{2n}(X,M^n)\rar&\Ext_{\La}^{2n}(X,N)\rar&\cdots.
\end{tikzcd}\]
\item [$(e)$]  For each $X\in\CC$ and for each
\[
0\rt L\rt M^1\rt \cdots\rt M^n\rt N\rt0
\]
exact sequence in $\mmod \La$ whose terms lie in $\CC$ there is an exact sequence
\[\begin{tikzcd}[column sep=tiny, row sep=tiny]
0\rar&\Hom_{\La}(N,X)\rar&\Hom_{\La}(M^n,X)\rar&\cdots\rar&\Hom_{\La}(M^1,X)\rar&\Hom_{\La}(L,X)\rar&{}\\        {}\rar&\Ext_{\La}^n(N,X)\rar&\Ext_{\La}^n(M^n,X)\rar&\cdots\rar&\Ext_{\La}^n(M^1,X)\rar&\Ext_{\La}^n(L,X)\rar&{}\\
{}\rar&\Ext_{\La}^{2n}(N,X)\rar&\Ext_{\La}^{2n}(M^n,X)\rar&\cdots\rar&\Ext_{\La}^{2n}(M^1,X)\rar&\Ext_{\La}^{2n}(L,X)\rar&\cdots.
\end{tikzcd}\]
\end{itemize}
\end{theorem}

An $n$-cluster tilting subcategory of $\mmod\La$ satisfying one, and hence all, of the above equivalent conditions, is called an $n\Z$-cluster tilting subcategory.

The notion of an $n$-cluster tilting subcategory, resp. $n\Z$-cluster tilting subcategory, of a triangulated category is implicit in \cite{KR}, \cite{I2} and \cite{GKO}. We use these notions throughout the paper without further references and refer the reader to the above mentioned references for definition and properties of these categories.

We also need the notions of $n$-exact and Frobenius $n$-exact categories. We refer the reader to \cite{Ja} for definition and propositions of $n$-exact categories and  just recall the definition of Frobenius $n$-exact categories.

 Let $(\CC,\CS)$ be an $n$-exact category. An object $I \in \CC$ is called $\CS$-injective if for every admissible monomorphism $f:A \rt B$, the sequence $\CC(B, I)\st{\CC(f, I)}\lrt \CC(A, I) \lrt 0$ is exact. Denote by $\CI$ the subcategory of $\CS$-injectives. We say that $(\CC, \CS)$ has enough $\CS$-injective if for every object $A \in \CC$, there exits an admissible $n$-exact sequence  $$A \rightarrowtail I^1 \rt I^2 \rt \cdots \rt I^n\twoheadrightarrow B$$ with $I^i \in \CI.$  The object $B$ does not depend on the choice of the $I^i$ in $\overline{\CC}$, so deonte it by $\mho^n(A)$. We can define the notion of $\CS$-projective and  having enough  $\CS$-projectives dually. Denote by $\CP$ the subcategory of $\CS$-projectives. We say that  $n$-exact category $(\CC, \CS)$ is Frobenius if it has enough $\CS$-injectives, enough $\CS$-projectives and $\CP=\CI$.\\
 Let $(\CC, \CS)$ be a Frobenius $n$-exact category. We define a class $\zeta=\zeta(\CS)$ of $n$-$\mho^n$-sequences in $\underline{\CC}$ as follows. Let $C^0 \st{f^0}\rt C^1$ be a morphism in $\CC$. Then, for every morphism of $n$-exact sequences of the form
 \[
 \begin{tikzcd}
 C^0\rar[tail]\dar{f^0}&I^1\rar\dar&\cdots\rar&I^n\rar[two heads]\dar&\mho^n(C^0)\dar[equals]\\
C^1\rar[tail]{f^1}&C^2\rar{f^2}&\cdots\rar{f^{n}}&C^{n+1}\rar[two
 heads]{f^{n+1}}&\mho^n(C^0)
 \end{tikzcd}
 \]
 the sequence
 \[
 \begin{tikzcd}
 C^0\rar{\underline{f^0}}&C^1\rar{\underline{f}^1}&C^2\rar{\underline{f}^2}&\cdots\rar{\underline{f}^n}&C^{n+1}\rar{\underline{f}^{n+1}}&\mho^n
 (C^0)
 \end{tikzcd}
 \]
 is called a \emph{standard $(n+2)$-angle}. An $n$-$\mho$-sequence
 $Y$ in $\underline{\CC}$ belongs to $\zeta$ if and only if it is isomorphic to a
 standard $(n+2)$-angle. In  \cite[Theorem 5.11]{Ja} is proved  that $(\underline{\CC}, \mho^n, \zeta(\CS))$ provides an $(n+2)$-angulated structure on $\underline{\CC}$, which is indeed a higher analog of \cite[Theorem I.2.6]{Ha}.

\section{$n\Z$-Gorenstein projective objects of $n\Z$-cluster tilting subcategories}	
Auslander and Bridger \cite{AB} introduced the class of modules of Gorenstein dimension zero as a natural generalization of the class of finitely generated projective modules over commutative noetherian rings. Encohs and Jenda \cite{EJ} generalized the notion to arbitrary modules over arbitrary rings called them Gorenstein projective modules. They also dualised the notion and introduced the class of Gorenstein injective modules. Based on these two classes of modules, they developed a
relative version of homological algebra, nowadays known as the Gorenstein homological algebra. It has been proved that these classes are worth to study, according to many applications they already have had.

Our aim in this section, is to introduce the notion of $n\Z$-Gorenstein projective objects in $n\Z$-cluster tilting subcategories. Throughout the section,  $\CC$ is an $n\Z$-cluster tilting subcategory of $\mmod \La.$

\begin{definition}
An acyclic complex
\[{\bf P}: \cdots \rt P^{-1} \rt P^0\rt P^1 \rt  \cdots 	\]
of finitely generated projective modules is called  $n\Z$-totally acyclic if
\begin{itemize}
\item [$(i)$] For every projective $\La$-module $Q$ and every exact sequence
\[{\bf Z}^i:\ \ 0 \rt Z^{in} \rt P^{in} \rt P^{in+1}\rt \cdots \rt P^{(i+1)n-1}\rt Z^{(i+1)n}\rt 0 ,\]
the induced complex $\Hom _\La({\bf Z}^i,Q)$ is acyclic, for all $i \in \Z$, where for every $j \in \Z, \ Z^j=\Ker(P^{j} \rt P^{j+1})$.
\item [$(ii)$] $Z^{in}\in \CC$, for all $i \in \Z$.
\end{itemize}
\end{definition}

Note that the above complex ${\bf P}$ is just a complete projective resolution in the language of \cite{EJ} with the extra assumption that all $ni$-cycles are in $\CC$.

\begin{definition}\label{DefnZGP}
A module $G$ in an $n\Z$-cluster tilting subcategory $\CC$ is called $n\Z$-Gorenstein projective if it is isomorphic to $Z^{in}$, for some $i \in \Z$ and  some $n\Z$-totally acyclic complex $\bf P$. We denote  by $\nGprj\CC$ the full subcategory of all $n\Z$-Gorenstein projective modules.
\end{definition}

It is clear that for $n=1$,  $\CC$ is just $\mmod\La$ itself and the subcategory $1\Z\mbox{-}\Gprj\mmod \La$ is nothing but the subcategory of all Gorenstein projective modules. We denote the subcategory of Gorenstein projective modules in $\mmod \La$  by $\Gprj\La$.

\begin{remark}
The notion of $n\Z$-Gorenstein injective modules in $\CC$ can be defined dually, and the full subcategory of all $n\Z$-Gorenstein injective modules in $\CC$ is  denoted by $n\Z\mbox{-}\Ginj\CC$. The subcategory $D(\CC)$ of $\mmod\La^{\op}$ is again an $n\Z$-cluster tilting subcategory and one can observe that
$\nGinj\CC\simeq \nGprj D(\CC)^{\rm{op}}$. So we restrict our study to the class of $n\Z$-Gorenstein projective modules.
\end{remark}

Following proposition collects some of the basic properties of $n\Z$-Gorenstein projective modules. The proof of the proposition is just based on the definition and is straightforward. So we skip the proof.

\begin{proposition}\label{Prop 3.6}
Let $\CC$ be an $n\Z$-cluster tilting subcategory and  $M \in \nGprj\CC$. Then the following properties hold.
\begin{itemize}
\item [$(i)$] $\prj\La \subseteq \nGprj\CC \subseteq  \CC \cap \Gprj\La$.
\item [$(ii)$] $\nGprj\CC$ is closed under finite direct sums.
\item [$(iii)$] For every $i \geqslant 1$ and every $P \in \prj\La,$ $\Ext^{in}_{\La}(M, P)=0$.
\item [$(iv)$] For every $i \geqslant 1$, $\Om^{in}(M) \in \nGprj\CC$.
\item [$(v)$] The $n$-the syzygy functor induces an equivalence $\Om^n: \underline{\nGprj}\CC \rt \underline{\nGprj}\CC$.	
\end{itemize}
\end{proposition}

\begin{remark}
\begin{itemize}
\item [$(i)$] Recall that an artin algebra $\La$ is called CM-free if $\Gprj\La=\prj\La$. Let $\La$ be an $n$-representation-finite algebra in the sense of \cite{IO} and $\CM(\La)$ be the unique $n$-cluster tilting subcategory of $\mmod\La$. Then it is not difficult to see that $\nGprj\CM(\La)=\prj\La.$ That is $\CM(\La)$ is $n\mbox{-}{\rm CM}$-free.
\item [$(ii)$] Over a self-injective algebra $\La$, $\Gprj\La=\mmod \La$. Similarly, if $\La$ is a self-injective algebra, for every $n\Z$-cluster tilting subcategory $\CC$,  $\nGprj\CC=\CC$.
\end{itemize}
\end{remark}

\begin{definition}\label{DefnZGCTS}
Let $\CC$ be an $n\Z$-cluster tilting subcategory of $\mmod\La.$
\begin{itemize}
\item[$(i)$] We say that $\CC$ is $n\Z$-Gorenstein projective cluster tilting, for simplicity  $n\Z$-$\Gprj$cluster tilting, if for every object $C \in \CC$, there exists a non-negative integer $i$ such that $\Om^{in}(C) \in \nGprj\CC.$
\item[$(ii)$] We say that $\CC$ is $n\Z$-Gorenstein injective cluster tilting, for simplicity $n\Z$-$\Ginj$cluster tilting, if for every object $C \in \CC$, there exists a non-negative integer $i$ such that $\Om^{-in}(C) \in \nGinj\CC.$
\item[$(iii)$] We say that $\CC$ is $n\Z$-Gorenstein cluster tilting, for simplicity $n\Z$-G-cluster tilting, if it is both $n\Z$-Gorenstein projective and $n\Z$-Gorenstein injective cluster tilting.
\end{itemize}
\end{definition}

Recall that an artin algebra $\La$ is called Iwanaga-Gorenstein, or simply Gorenstein,  if  both $\id_{\La}\La$, the self-injective dimension of $\La$ as a right $\La$-module and $\id_{{\La}^{\op}}\La$,  the self-injective dimension of $\La$ as a left $\La$-module,  are finite. It is known that in this case $\id_{\La}\La=\id_{{\La}^{\op}}\La$, say is equal to $d \in \N$. Then $\La$ is also called a $d$-Gorenstein algebra. In the following we show that there is a nice connection between the $n\Z$-Gorenstein cluster tilting subcategories and Gorenstein algebras. More explicitly, we prove the following theorem.

\begin{theorem}\label{3.10}
Let $\CC$ be an $n\Z$-cluster tilting subcategory of $\mmod\La.$ Then $\CC$ is an $n\Z$-G-cluster tilting subcategory of $\mmod \La$ if and only if $\La$ is a Gorenstein algebra.
\end{theorem}

The proof of the theorem is proceeded by the following two propositions.

\begin{proposition}\label{GorSymmetryConj}
Let $\CC$ be an $n\Z$-cluster tilting subcategory of $\mmod\La.$
\begin{itemize}
\item [$(i)$] If $\CC$ is $n\Z$-$\Gprj$cluster tilting, then $\id_{{\La}^{\op}}\La$ is finite.
\item [$(ii)$] If $\CC$ is $n\Z$-$\Ginj$cluster tilting, then $\id_{\La}\La$ is finite.
\item [$(iii)$] If $\CC$ is $n\Z$-G-cluster tilting, then $\La$ is a Gorenstein algebra.
\end{itemize}
\end{proposition}

\begin{proof}
$(i)$  We prove that $\rm{pd}_{\La} \ D({}_{\La}\La)$ is finite.  Since $D({}_{\La}\La)$ is in $\CC$,   there exists a non-negetive integer $i$ such that $\Om^{in}(D({}_{\La}\La)) \in \nGprj\CC.$ So there exists an $n\Z$-totally acyclic complex
\[ 0 \rt \Om^{in}(D({}_{\La}\La)) \rt Q^0 \rt \cdots \rt Q^{in-1} \rt H \rt 0,\]
 where $Q^j$ for $j\in\lbrace 0,...,in-1\rbrace$ is projective and $H\in\nGprj\CC.$ Let $L$  be the cokernel of $0 \rt \Om^{in}(D({}_{\La}\La)) \rt Q^0 $.  Using dimension-shifting argument   we get
 $$\Ext^1_{\La}(L, \Om^{in}(D({}_{\La}\La)))\simeq \Ext^{in+1}_{\La}(H, \Om^{in}(D({}_{\La}\La)))\simeq
 \Ext^1_{\La}(H,D({}_{\La}\La))=0.$$
  This implies that $\Om^{in}(D({}_{\La}\La)$ is projective. Hence  $\rm{pd}_{\La} \ D({}_{\La}\La) < \infty$.

$(ii)$ The proof follows by the dual argument of the proof of part $(i)$.

$(iii)$ Immediate consequence of parts $(i)$ and $(ii)$.
\end{proof}

\begin{proposition}\label{TheIwangaGoreNZcluster}
Let $\La$ be a Gorenstein algebra and $\CC$ be an $n\Z$-cluster tilting subcategory of $\mmod\La.$ Then $\CC$ is an $n\Z$-G-cluster tilting subcategory of $\mmod \La.$ Moreover, in this case we have $\nGprj\CC=\Gprj\La \cap \CC$ and $\nGinj\CC=\Ginj\La \cap \CC.$
\end{proposition}

\begin{proof}
We first show that $\nGprj\CC=\Gprj\La \cap \CC$. By Proposition \ref{Prop 3.6}, we have  inclusion $\nGprj\CC \subseteq \Gprj\La\cap \CC$. Now let $G \in \Gprj\La\cap \CC.$  Since  $G$ is a Gorenstein projective module, by defintion,  there exits a totally acyclic complex
 \[{\bf P}: \ \ \cdots \rt P^{-1}\st{d^{-1}}\rt P^0 \st{d^0}\rt P^1 \rt \cdots \]
with $G=\rm{Ker}\ d^0.$ We claim that $\bf{P}$ is an $n\Z$-totally acyclic complex. Since $G\in \CC$, all $ni$-cycles for $i<0$, are in $\CC$. It remains to show that all $ni$-cycles for $i>0$, are in $\CC$. Let $i>0$ and $Z^{in}={\rm{Ker}}(P^{in}\longrightarrow P^{in+1})$. We show that $Z^{in}\in \CC$. Since $\La$ is Gorenstein, we can choose a positive integer $k$ such that $\Om^{kn}(Z^{in}) \in \Gprj\La \cap \CC$. Now for all $j\in\lbrace 1,...,n-1\rbrace$ and all $C\in \CC$,
$$\Ext^j_{\La}(Z^{in}, C)\simeq \Ext^{j+kn}_{\La}(Z^{in}, \Om^{kn}(C) )\simeq \Ext^j(\Om^{kn}(Z^{in}), \Om^{kn}(C))= 0.$$
So $\Z^{in}\in\CC$ by definition, as it was claimed.

Now we show that $\CC$ is an $n\Z$-$G$-cluster tilting subcategory. We only prove that $\CC$ is an $n\Z$-$\Gprj$cluster tilting subcategory of $\mmod\La$. The proof of the fact that $\CC$ is also an  $n\Z$-$\Ginj$cluster tilting subcategory followes by duality. To this end, we show that for every $C\in\CC$, there exists a non-negetive integer $i$ such that $\Om^{in}(C) \in \nGprj\CC$. Since $\La$ is a Gorenstein algebra,  there exists an integer  $m \in \N$ such that $\Om^{m}(\mmod\La)=\Gprj\La$. So for $C\in \CC$, there exists non-negeive integer $i$ such that $\Om^{in}(C)\in \Gprj\La \cap \CC$. Hence by the first part of the proof
$\Om^{in}(C)\in \nGprj\CC$.
\end{proof}

\textit{Proof of  Theorem \ref{3.10}.}
 It follows as an immediate consequence of the above two propositions.

\begin{proposition}\label{GProjApproxiamation}
Let $\CC$ be an $n\Z$-cluster tilting subcategory of $\mmod \La.$
\begin{itemize}
\item[$(i)$] If $\CC$ is an $n\Z$-$\Gprj$cluster tilting subcategory of $\mmod\La,$  then for every $C \in \CC$, there is a short exact sequence
\[0 \rt M \rt G \rt C \rt 0\] such that $G \in \nGprj\CC$ and $\pd_{\La}M < \infty$.  In this case, we say that $C$ admits an $n\Z$-Gorenstein projective precover in $\CC.$
\item[$(ii)$]	If $\CC$ is  an $n\Z$-$\Ginj$cluster tilting subcategory of $\mmod\La,$ then for every $C \in \CC$, there is a short exact sequence
\[0 \rt C \rt E \rt N \rt 0\] such that $E \in \nGinj\CC$ and $\id_{\La}N < \infty$. In this case, we say that $C$ admits an $n\Z$-Gorenstein injective preenvelope in $\CC.$
\end{itemize}
\end{proposition}

\begin{proof}
The idea of the proof is due to Auslander-Buchweitz known as pitchfork construction, see\cite{ABu}. We only prove the statement $(i)$. The statement $(ii)$ follows similarly, or rather dually. Let $C\in \CC.$ There exists  a non-negative integer  $i$ such that  $\Om^{in}(C)\in\nGprj\CC$. So there exists an exact sequence
\[0 \rt \Om^{in}(C) \rt P^{in-1} \rt \cdots \rt P^0 \rt C \rt 0\]
where  $P^j$ is projective, for all $j \in \lbrace 0, ... , in{\mbox{-}}1\rbrace$ and $\Om^{in}(C) \in \nGprj\CC.$ Since  $\Om^{in}(C)\in \nGprj\CC$, there exists  a $\Hom_{\La}(-,\prj\La)$-exact  exact sequence
\[0 \rt \Om^{in}(C) \rt Q^0 \rt \cdots \rt Q^{in-1} \rt S \rt 0,\]
where $Q^j$ is  projective, for all $j \in \lbrace 0, ... , in{\mbox{-}}1\rbrace$ and  $S \in \nGprj\CC$. So we get the  following commutative diagram
	
\begin{equation*}
\label{eq:n-exact-sequence-induce-n+2-angle-diagram}
\begin{tikzcd}
	{\bf P}\dar{f}&0 \rar  &\Om^{in}(C)\rar\dar[equals]&Q^0\rar\dar&\cdots\rar&Q^{in-1}\rar\dar&S\dar\rar& 0\\
	\mathbf{Q}&0\rar &\Om^{in}(C)\rar&P^{in-1}\rar&\cdots\rar&P^0\rar&C \rar & 0.
	\end{tikzcd}
\end{equation*}
	
This diagram gives a chain map between complexes,
\begin{equation*}
\xymatrix{ 0\ar[r]&  Q^0 \ar[r] \ar[d]&  \cdots \ar[r] &  Q^{in-1} \ar[r] \ar[d]&  S\ar[d]\ar[r]&0\\ \ \ 0\ar[r]& P^{in-1} \ar[r] & \cdots \ar[r] & P^0\ar[r] & C\ar[r]&0,}
\end{equation*}
which is quasi-isomorphism. Hence its mapping cone
$$0 \rt Q^0 \rt P^{in-1}\oplus Q^1 \rt \cdots \rt P^0 \oplus S \st {\phi}\rt C \rt 0,$$
is exact. So we get a short exact sequence $0 \rt \Ker \phi \rt P^0 \oplus S \st{\phi}\rt C \rt 0$ which is the desired sequence, as $P^0\oplus S\in \nGprj\CC $ and  $\rm{pd}_{\La}\ \Ker \phi < \infty.$ This completes the proof of the first assertion.

The latter assertion follows  from the first part, using the inclusion $\nGprj\CC\subseteq   \Gprj\La$, which establishes the existance of the  epimorphism $G \rt C$ as a precover.
\end{proof}

\begin{lemma}\label{ContNGor}
Let $\CC$ be an $n\Z$-cluster tilting subcategory of $\mmod \La.$
\begin{itemize}
\item[$(i)$]	If $\CC$ is  $n\Z$-$\Gprj$cluster tilting, then $\nGprj\CC$ is contravariantly finite in $\Gprj\La.$
\item[$(ii)$]  If $\CC$ is  $n\Z$-$\Ginj$cluster tilting, then $\nGinj\CC$ is covariantly finite in $\Ginj\La.$
\end{itemize}
\end{lemma}

\begin{proof}
We only prove part $(i)$.  The statement $(ii)$ follows similarly.
Let $G \in\Gprj\La$.   Since $\CC$ is functorially finite in $\mmod \La$, then there exits a right $\CC$-approximation $C \st{f}{\rt} G,$ such that $C\in\CC$. By Proposition \ref{GProjApproxiamation}, there is a short exact sequence $0 \rt L \rt G' \st{g}\rt C \rt 0$ with $G' \in \nGprj\CC$ and $\rm{pd}_{\La} \ L < \infty.$ Clearly, $g$ is a right $\nGprj\CC$-approximation of $C\in\CC.$  Now  it is not difficult to see that the map $G' \st{fg}\rt G$ is a right  $\nGprj\CC$-approximation of $G$ in $\Gprj\La.$  Hence $\nGprj\CC$ is contravariantly finite in $\Gprj\La$.	
\end{proof}

It is known that the functors $\tau_n: \underline{\CC} \rt \overline{\CC}$ and $\tau^{-1}_n:\overline{\CC} \rt \underline{\CC}$  are  mutually quasi-inverse equivalences, where $\tau_n=\tau \circ \Om^{n-1}$, $\tau^{-1}_n=\tau^{-1}\circ \Om^{-(n-1)}$, $\Om:\underline{\rm{mod}}\mbox{-}\La \rt \underline{\rm{mod}}\mbox{-}\La$ is the syzygy functor, $\Om^{-1}:\overline{\rm{mod}}\mbox{-}\La \rt \overline{\rm{mod}}\mbox{-}\La$ is the cosyzygy functor and finally $\tau$ and $\tau^{-}$ are the usual Auslander-Reiten translations, see \cite{I} and \cite{I1}.

\begin{proposition}\label{NauslanderReitenNGor}
Let $\CC$ be an $n\Z$-$G$-cluster tilting subcategory of $\mmod \La$. Then there exists the following commutative diagram
\[\xymatrix{ \underline{\CC} \ar[r]^{\tau_n} &  \overline{\CC} \ar@/^0.5pc/[l]^{\tau_n^{-1}} \\  \underline{\nGprj}\CC \ar@{^{(}->}[u] \ar[r]^{\tau_n\mid} & \overline{\nGinj}\CC \ar@{^{(}->}[u] \ar@/^0.5pc/[l]^{\tau_n^{-1}\mid}}\]	
 with equivalences in rows, where $\tau_n$ and $\tau^{-1}_n$ denote the $n$-Auslander-Reiten translations on $\CC$.
\end{proposition}

\begin{proof}
It is enough to show that $n$-Auslander-Reiten translations can be restricted to arise equivalences in the bottom row.
Let $G\in \nGprj\CC$. Since $\tau_n(G)\in\CC$, by Proposition \ref{TheIwangaGoreNZcluster}, it suffices to show that  $\tau_n(G)$ is Gorenstein injective. This follows from the facts that syzygy functor maps Gorenstein projectives to Gorenstein projectives and $\tau$ maps Gorenstein projectives to Gorenstein injectives, see
\cite[Proposition 2.2.7]{C}.
\end{proof}

Since $\Gprj\La$ is  an extension closed subcategory of $\mmod \La$, it inherits an exact  structure making the inclusion $\Gprj\La \hookrightarrow \mmod \La$ into an exact functor. It is known that, under this exact structure, $\Gprj\La$ becomes a Frobenius exact category with projective-injective objects coincide with $\prj\La.$ So the stable category $\underline{\Gprj}\La$ has a natural structure of a triangulated category, see \cite[Chapter 1]{Ha}. In fact, a short exact sequence $0 \rt G^1\st{f} \rt G^2 \st{g}\rt G^3 \rt 0$ with all terms in $\Gprj\La$ induces the triangle $G^1 \st{\underline{f}}\rt G^2 \st{\underline{g}}\rt G^3 \st{\underline{h}}\rt \Sigma G^1$ in $\underline{\Gprj}\La$ , and conversely  any triangle in $\underline{\Gprj}\La$ is obtained in this way; see also Lemma 1.2 of \cite{CZ}.

\begin{corollary}\label{FnctoriallyFiniNGproj}
Let $\La$ be a Gorenstein algebra and $\CC$ be an $n\Z$-cluster tilting subcategory of $\mmod\La.$ Then the following hold.
\begin{itemize}
\item[$(i)$]  $\underline{\nGprj}\CC$ is functorialy finite in $\underline{\Gprj}\La$.
\item[$(ii)$] Dually, $\overline{\nGinj}\CC$ is functorialy finite in $\overline{\Ginj}\La.$
\end{itemize}	
\end{corollary}

\begin{proof}
We only prove $(i)$. The statement $(ii)$ follows dually. According to Lemma \ref{ContNGor}, $\nGprj\CC$ is contravariantly finite  in $\Gprj\La$. This, in turn, implies that $\underline{\nGprj}\CC$  is contravariantly finite in $\underline{\Gprj}\La.$ So  it remains to show that  $\underline{\nGprj}\CC$ is covariantly finite in $\underline{\Gprj}\La$.
Using Proposition \ref{GProjApproxiamation}, one can see that $\overline{\nGinj}\CC$ is covariantly finite in $\overline{\CC}.$ Hence in view of  Proposition \ref{NauslanderReitenNGor}, $\underline{\nGprj}\CC$ is covariantly finite in $\underline{\CC}$. Let $G\in\underline{\Gprj}\La$. Let $G\st{\underline{g}} \rt C $ and $C \st{\underline{f}} \rt G'$ be the left $\underline{C}$-approximation of $G$ and the  left $\underline{\nGprj}\CC$-approximation of $C$ in $\underline{\CC}$, respectively.  It is not difficult to see that $G \st{\underline{fg}} \rt G'$ is left $\underline{\nGprj}\CC$-approximation of $G$ in $\underline{\Gprj}\La$.
\end{proof}

The following result recover \cite[Corollary 7.3]{K}  by using the language of $n\Z$-Gorenstein projective modules.

\begin{proposition}\label{LeftRightOrtogonal}
Let $\La$ be a Gorenstein algebra and $\CC$ be an $n\Z$-cluster tilting subcategory of $\mod\La$. Then
$\nGprj\CC={}^{\bot_{n}}(\nGprj\CC)\cap\Gprj\La=(\nGprj\CC)^{\bot_{n}}\cap \Gprj\La.$
\end{proposition}

\begin{proof}
First we show that
$$\nGprj\CC={}^{\bot_{n}}(\nGprj\CC)\cap\Gprj\La.$$
Clearly $\nGprj\CC\subseteq {}^{\bot_{n}}(\nGprj\CC)\cap \Gprj\La$. For the converse, let $X\in {}^{\bot_{n}}(\nGprj\CC)\cap\Gprj\La$. Since $\nGprj\CC=\Gprj\La\cap\CC$, it is enough to show that $X\in \CC$. By Proposition \ref{GProjApproxiamation}, for every $C\in\CC$, there is a short exact sequence $0\longrightarrow L\longrightarrow G\longrightarrow C\longrightarrow 0$, such that $G\in\nGprj\CC$ and $\rm{pd}_{\La}  L < \infty.$ By applying $\Hom_{\La}(X, -)$
on the above short exact sequence, we get the exact sequence  $$\Ext^i_{\La}(X,G)\longrightarrow \Ext^i_{\La}(X,C)\longrightarrow\Ext^{i+1}_{\La}(X,L).$$
Note that since $\rm{pd}_{\La} L< \infty$, for every $i>1$, $\Ext^i(X,L)=0$. Also by assumption for all $i\in{\lbrace1,...,n-1\rbrace}$, $\Ext^{i}_{\La}(X,G)=0$. So for all $i\in\lbrace 1,...,n-1\rbrace$, $\Ext^i_{\La}(X,C)=0$. Hence $X\in\CC$. This complete the proof of the  first equality.
Now we show that $\nGprj\CC=(\nGprj\CC)^{\bot_{n}}\cap \Gprj\La$.  Clearly $\nGprj\CC\subseteq (\nGprj\CC)^{\bot_{n}}\cap \Gprj\La$.
For the reverse inclusion, let $X\in(\nGprj\CC)^{\bot_{n}}\cap \Gprj\La$. By  \cite[Theorem 2.2.3]{I}, there exists a minimal $\CC$-resolution
$$\epsilon: 0 \rt X \rt C_1 \st{d_1}\rt \cdots\rt C_{n-1} \st{d_{n-1}} \rt C_n \rt 0,$$
where   $C_i\in \CC.$ By induction on $n$, we show that $X\in\CC$. This implies that $X\in\nGprj\CC$. The case $n=1$ is trivial. So assume that $n=2$. By Proposition  \ref{GProjApproxiamation}, there exists a short exact sequence
 $0 \rt \Ker f_2 \rt G_2 \st{f_2}\rt C_2 \rt 0$, where $f_2$ is a right minimal $\Gprj\La$-approximation, $G_2\in\nGprj\CC$ and $\pd_{\La} \Ker f_2<\infty$. Consider the following pull-back diagram
 \[\xymatrix{
& & & 0 \ar[d]  & 0 \ar[d] &  \\
& & & \Ker f_2 \ar[d] \ar@{=}[r] & \Ker f_2 \ar[d] & \\
& 0 \ar[r] & X\ar@{=}[d] \ar[r] & U \ar[d]^g \ar[r] & G_2 \ar[d]^{f_2} \ar[r] & 0 \\
& 0 \ar[r] & X\ar[r] & C_1 \ar[d]  \ar[r] & C_2 \ar[r]\ar[d] & 0 \\
& & & 0 &  0 & }\]
 Since $\Ext^1_\La(G_2,X)=0$, the middle row in the above diagram splits. Hence $U\simeq G_2\oplus X$. Since $\Ker f_{2}=\Ker g$, $\pd_{\La} \Ker g < \infty$ and hence $g$ is a right $\Gprj\La$-approximation. On the other
 hand, since $C_1\in\CC$, there exists a right minimal  $\Gprj\La$-approximation $G_1\rt  C_1$ such that $G_1\in\nGprj\CC$. Set $X=G\oplus H$, where $G\in\nGprj\CC$ and $H$ has no non-zero summand in $\nGprj\CC$. By the minimal property of $G_1\rt  C_1$, $H$ should be a direct summand of $\Ker f_2$. Hence $\pd_\La H< \infty$. This implies that $H$ is projective, as it is Gorenstein projective of finite projective dimension.
 Hence $H=0$ and $X\in \CC$.
 Now assume inductively that $n > 2$. Put $\Ker d_i=K_i$, for $i\in \lbrace 2,...,n-1\rbrace$. We claim that for every $i\in\lbrace 2,...,n{\mbox{-}}1\rbrace$, there exists a right $\Gprj\La$-approximation $G^{\prime}_{i}\rt K_i$ such that $G^{\prime}_{i}\in \nGprj\CC$ and its kernel is of finite pojective dimension.

 If we prove the claim, in particular, for $K_2$ we get a right $\Gprj\La$-approximation $G^{\prime}_{2}\st{f^{\prime}_2}\rt K_2$ with $G^{\prime}_{2}\in\nGprj\CC$ and $\pd_\La\Ker f_2< \infty$. Now by considering pull-back diagram of short exact sequence $0\rt X \rt C_1\rt K_2\rt 0$ along with $G^{\prime}_{2}\st{f^{\prime}_2}\rt K_2$, similar to the case $n=2$, we get the result.

For the proof of the claim, we use inverse induction on $i$. To do this, assume that $i=n-1$. By applying $\Hom_\La(\nGprj\CC,-)$ on the short exact sequence $0\rt K_{n-1}\rt C_{n-1}\rt C_n\rt 0,$ we see that this sequence  is  $\Hom_\La(\nGprj\CC,-)$-exact. This follows using the fact that $X \in (\nGprj\CC)^{\bot_{n}}$  and then applying a dimension-shifting argument to the corresponding short exact sequences of $\epsilon$.

In particular $\Ext^1_\La(G',K_{n-1})=0$, for every $G'\in\nGprj\CC$.  Consider the following pull-back diagram $(\dagger)$
\[\xymatrix{
& & & 0 \ar[d]  & 0 \ar[d] &  \\
& & & \Ker f_n \ar[d] \ar@{=}[r] & \Ker f_n \ar[d] & \\
& 0 \ar[r] & K_{n-1}\ar@{=}[d] \ar[r] & U \ar[d]^{g_n} \ar[r] & G_n \ar[d]^{f_n} \ar[r] & 0 \\
& 0 \ar[r] & K_{n-1}\ar[r] & C_{n-1} \ar[d]  \ar[r] & C_n \ar[r]\ar[d] & 0 \\
& & & 0 &  0 & }\]
where $f_n$ is a right $\nGprj\CC$-approximation of $C_n $ in $\CC$.  Since $\Ext^1_\La(G_n,K_{n-1})=0$, then $U=K_{n-1}\oplus G_n$. Now since the short exact sequence $0\rt\Ker f_n\rt U\rt C_{n-1}\rt 0$ is $\Hom_\La(\nGprj\CC,-)$-exact, we get the following commutative diagram with exact columns and rows.
\[\xymatrix{& 0 \ar[d] & 0 \ar@{-->}[d] & 0 \ar[d]& &\\
0 \ar@{-->}[r] & \Om(\Ker f_n) \ar[d] \ar@{-->}[r] & W_{n-1} \ar@{-->}[d] \ar@{-->}[r] & \rm{Ker}\ f_{n-1} \ar[d] \ar@{-->}[r] & 0\\
0 \ar[r] & P \ar[d] \ar[r] & P\bigoplus G_{n-1}\ar@{-->}[d] \ar[r] & G_{n-1} \ar[d]^{f_{n-1}} \ar[r] & 0\\
0 \ar[r] & \Ker f_n \ar[d] \ar[r] & K_{n-1}\oplus G_n\ar@{-->}[d] \ar[r] & C_{n-1} \ar[d] \ar[r] & 0\\
& 0  & 0  & 0 & }\]
The short exact sequence in the middle column, implies the first step of the induction.\\
Now suppose that the claim is proved for $i\in \lbrace 2,...,n-1\rbrace$. The proof for the case $i-1$ is analogous to the first step of the claim, we replace   the short exact sequence $0\rt\Ker f_n\rt G_n\rt C_n\rt 0$ in the diagram $(\dagger)$ with the short exact sequence $0\rt \Ker f^{\prime}_n\rt G^{\prime}_i\rt K_i\rt 0$, where $G^{\prime}_i\in\nGprj\CC$ and $\pd_\La \Ker f^{\prime}_i<\infty$.
\end{proof}
\begin{proposition}\label{StableNGprj}
Let $\La$  be a Gorenstein algebra and $\CC$ be an $n\Z$-cluster tilting subcategory of $\mod\La$. Then the subcategory $\underline{\nGprj}\CC$ is an $n$-cluster tilting subcategory of the triangulated category $\underline{\Gprj}\La$.
\end{proposition}
\begin{proof}
By definition, we need to show that $(i)$ $\underline{\nGprj}\CC$ is functorially finite in $\underline{\Gprj}\La$ and
\begin{align*}
(ii)\  \underline{\nGprj}\CC=&\{X\in \underline{\Gprj}\La \mid \forall i\in\{1,\dots,n-1\},\
\underline{\rm{Hom}}_{\La}(X, \Sigma^i\nGprj\CC)=0\}\\
=&\{X\in \underline{\Gprj}\La \mid \forall i\in\{1,\dots,n-1\},\ \underline{\rm{Hom}}_{\La}( \nGprj\CC, \Sigma^i X)=0\},
\end{align*}
where $\Sigma$ denotes the quasi-isomorphism of the syzygy functor $\Om:\underline{\Gprj}\La\rt \underline{\Gprj}\La.$\\
The fact that $\underline{\nGprj}\CC$ is functorially finite in $\underline{\Gprj}\La$ follows from Corollary  \ref{FnctoriallyFiniNGproj}. For $(ii)$ note that by \cite[Subsection 2.1]{I1}, for every $i\in\lbrace 1,...,n-1\rbrace$, we get the following isomorphisms
\[\underline{\rm{Hom}}_{\La}(X, \Sigma^i G')\simeq \underline{\rm{Hom}}_{\La}(\Om^i (X),  G') \simeq \Ext^i_\La(X, G'), \]
\[\underline{\rm{Hom}}_{\La}(G^{\prime}, \Sigma^i X)\simeq \underline{\rm{Hom}}_{\La}(\Om^i (G'), X) \simeq \Ext^i_\La(G', X), \]
for all $G'\in \nGprj\CC$ and $X\in\Gprj\La$. Hence the desired equalities follow from the above proposition.
\end{proof}

\begin{theorem}\label{NClusterTilting}
Let $\La$ be a Gorenstein algebra and $\CC$ be an $n\Z$-cluster tilting subcategory of $\mmod \La.$ The following hold.
\begin{itemize}
\item [$(i)$] The subcategory $\nGprj\CC$ is an $n\Z$-cluster tilting subcategory of $\Gprj\La.$
\item [$(ii)$] The subcategory $\underline{\nGprj}\CC$ is an $n\Z$-cluster tilting subcategory of $\underline{\Gprj}\La.$		
\end{itemize}
\end{theorem}

\begin{proof}
 $(i)$ By Proposition \ref{LeftRightOrtogonal}, $\nGprj\CC={}^{\bot_{n}}(\nGprj\CC)\cap\Gprj\La=(\nGprj\CC)^{\bot_{n}}\cap \Gprj\La$.  Since $\CC$ an $n\Z$-cluster tilting subcategory of $\mmod\La$, we have $\Ext^i_{\La}(\nGprj\CC, \nGprj\CC)=0$, for  all $i \notin n\Z$. So to complete the proof, it remains to prove that $\nGprj\CC$ is a functorially finite subcategory of $\Gprj\La$. To see this, we note that by Proposition \ref{StableNGprj}, $\underline{\nGprj}\CC$ is an $n$-cluster tilting subcategory of  $\underline{\Gprj}\La$, so  \cite[Theorem 11.3 (b)(ii)]{B} implies that $\nGprj\CC$ is a functorially finite subcategory of $\Gprj\La$.//
  $(ii)$ Follows from Proposition \ref{StableNGprj} and $(i)$.
\end{proof}

Since $\underline{\nGprj}\CC$ is an $n$-cluster tilting  subcategory of $\underline{\Gprj}\La$ which is stable under $n$-th power of the syzygy functor, then as an application of the above theorem, we obtain from \cite[Theorem 1]{GKO} that $\underline{\nGprj}\CC$ has the  structure of an $(n+2)$-angulated category. That is, we have the following corollary.

\begin{corollary}\label{Nangulatedcategory}
Let $\La$ be a Gorenstein algebra and $\CC$ be an $n\Z$-cluster tilting subcategory of $\mmod \La.$ Then $(\underline{\nGprj}\CC, \Sigma^n, \bigtriangleup) $
is an $(n+2)$-angulated category, where $\bigtriangleup$ is the class of all sequences \[ X^1\st{\underline{f}^1}\rt X^2\st{\underline{f}^2}\rt\cdots\st{\underline{f}^{n+1}}\rt X^{n+2}\st{\underline{f}^{n+2}}\rt  \Sigma^{n} X^1 \]
in $\underline{\nGprj}\CC$ such that there exists a diagram
\[ \xymatrix@!=.5pc{ & X^2 \ar[rr]^{\underline{f}^2} \ar[rd] && X^3 \ar[rd] && \cdots && X^{n+1} \ar[rd]^{\underline{f}^{n+1}} \\ X^1 \ar[ru]^{\underline{f}^1} \ar@{<-{|-}}[rr] && X^{2.5} \ar[ru] \ar@{<-{|-}} [rr] && X^{3.5} & \cdots & X^{n.5} \ar[ru] \ar@{<-{|-}}[rr] && X^{n+2}} \]
 where, all oriented  triangles are induced by short exact sequences in $\Gprj\La$, all non-oriented triangle commutes, and $\underline{f}^{n+2}$ is the composition along the lower edge of the diagram.
\end{corollary}

The above theorem, in particular, implies that the subcategory $\nGprj\CC$ is a Frobenius $n$-exact category.

\begin{theorem}
Let $\La$ be a Gorenstein algebra and $\CC$ be an $n\Z$-cluster tilting subcategory of $\mmod\La.$ Then
\begin{itemize}
\item[$(i)$] The pair $(\nGprj\CC, \CS)$ is a Frobenius $n$-exact category. Here $\CS$ denotes all exact sequences
\[0 \rt G^0 \rt G^1 \rt \cdots \rt G^n\rt G^{n+1}\rt 0\]
with terms in $\nGprj\CC$. Note that by \cite[Theorem 4.14]{Ja}, $\CS$ gives an $n$-exact structure on $\nGprj\CC$;
\item[$(ii)$] Let $(\underline{\Gprj}\La, \Sigma_{\underline{\Gprj}\La},\triangle)$ be the canonical triangulated structure of $\underline{\Gprj}\La$ and
\\ $(\underline{\nGprj}\CC, \mho^n, \zeta)$ be  the standard $(n+2)$-angulated structure of $\underline{\nGprj}\CC$, which have been already defined. Then, we have an equivalence of $(n+2)$-angulated categories between $(\underline{\nGprj}\CC, \mho^n, \zeta)$ and $(\underline{\nGprj}\CC, \Sigma^{n}_{\underline{\Gprj}\La},\CS)$. The latter $(n+2)$-angulated structure comes from this fact that $\underline{\nGprj}\CC$ is an $n$-cluster tilting subcategory of $\underline{\Gprj}\La.$
\end{itemize}	
\end{theorem}

\begin{proof}
$(i)$ We have observed by Theorem \ref{NClusterTilting} that  the subcategory $\nGprj\CC$ of  Frobenius exact category $\Gprj\La$ is $n$-cluster tilting subcategory with $\mho^n(\nGprj\CC) \subseteq \nGprj\CC.$  Therefore, by \cite[Theorem 5.16]{Ja} we get the result.

$(ii)$ Follows from part $(iii)$ of \cite[Theorem 5.16]{Ja}.
\end{proof}

\section{$n$-resolving subcategories}
Let $\CA$ be an abelian category with enough projective objects. One of the important properties of the subcategory $\Gprj\CA$ is its resolving property, i.e. it
contains all projective objects and moreover is closed under extensions and kernels of epimorphisms. Our aim in this section is to show that $\nGprj\CC$ has similar property, but in the higher context, where $\CC$ is an $n\Z$-cluster tilting subcategory of $\mmod\La$.

Let us begin by the definition of an $n$-resolving subcategory.

\begin{definition}\label{Resolving}
A full subcategory $\CM$ of an $n$-cluster tilting subcategory $\CC$ of $\mmod \La$ is called $n$-resolving if it satisfies the following conditions.
\begin{itemize}
\item [$(i)$] $\prj\La \subseteq \CM$.
\item [$(ii)$]For every epimorphism  $M \rt M' \rt 0$ in $\CM$ there exists  an $n$-exact sequence
\[0 \longrightarrow  M^1\rt  \cdots \rt  M^n \rt M\rt M' \rt 0 \]
 with all terms belong to $\CM.$
\item [$(iii)$] Each exact sequence
\[0 \rt M' \rt  A^1\rt  \cdots  \rt A^n \rt M \rt 0 \]  in $\mmod\La$,  with $M$ and $M'$ in $\CM$, is Yoneda equivalent to an $n$-exact sequence
\[0 \rt M' \rt  M^1\rt  \cdots  \rt M^n \rt M \rt 0 \]
with $M_i \in \CM$ for each $i\in\lbrace 1,...,n\rbrace.$
\end{itemize}
\end{definition}

Note that since $\CC$ contains $\prj\La$, it follows that  $\prj\CC=\prj\La$; see Subsection 3.2 of \cite{Ja} for definition of $\prj\CC$ of projective objects in an $n$-abelian category. So statement $(i)$ can be rewritten as $\prj\CC\subseteq \CM$. Moreover, in view of  \cite [A.1. Proposition ]{I}, we may assume that all middle terms of the first sequence of the statement $(iii)$, i.e. $A^i$, $i\in\lbrace 1,...,n\rbrace$, are in $\CC$.

Parts $(ii)$ and $(iii)$ of the  above definition are motivated by  \cite[Definition 2.11]{HJV} to give a `higher dimensional' version of the concepts `closed under epimorphisms' and `closed under extensions', respectively. Note that by \cite[Theorem 3.16]{Ja}, $n$-cluster tilting subcategory $\CC$ is an $n$-abelian category.

Following lemma  and propositions  prepare the ground for proving the main theorem of this section, i.e. Theorem \ref{NresolvingGP}.
\begin{lemma}\label{Condition(ii)}
Let $\La$ be a Gorenstein algebra and  $\CC$ be an $n\Z$-cluster tilting subcategory of $\mmod\La$. Then for every epimorphism $A\rt B$ in $\nGprj\CC$, there exists an exact sequence
\[0 \rt  X^1\rt  \cdots \rt  X^n \rt A \rt B \rt 0 \]
 with all terms belong to $\nGprj\CC$.
\end{lemma}

\begin{proof}
Let  $A \st{f} \rt B$ be an epimorphism in $\nGprj\CC$ . By Corollary \ref{Nangulatedcategory}, the morphism $\underline{f}$ can be completed  to an $(n+2)$-angle
\[	X^1 \st{\underline{f}^1}\rt X^2 \st{\underline{f}^{2}}\rt \cdots X^n \st{\underline{f}^{n}}\rt A \st{\underline{f}}\rt B  \st{\underline{h}}\rt \Sigma^{n} X^1 \]
in $\underline{\nGprj}\CC$ and so there exists a diagram
\[ \xymatrix@!=.5pc{ & X^2 \ar[rr]^{\underline{f}^2} \ar[rd]^{\underline{g}^2} && X^3 \ar[rd] && \cdots && A \ar[rd]^{\underline{f}} \\ X^1 \ar[ru]^{\underline{f}^1} \ar@{<-{|-}}[rr] && X^{2.5} \ar[ru]^{\underline{l}^2} \ar@{<-{|-}} [rr] && X^{3.5} & \cdots & X^{n.5} \ar[ru]^{\underline{l}^n} \ar@{<-{|-}}[rr] && B} \]
such that all the oriented  triangles  in the above diagram are induced by  short exact sequences in $\Gprj\La$. Now  by gluing the short exact sequences, we obtain the following exact sequence
$$0 \rt X^1 \st{f^1}\rt X^2\st{l^2g^2}\rt X^3\rt \cdots \rt X^n \st{l^ng^n} \rt A \st{f}\rt B \rt 0,$$
as required.
\end{proof}

\begin{remark}\label{ProjPresenCoker}
Let $\CX$ be an additive category and $\mmod \CX$  be the category of all finitely presented functors over $\CX.$ It is known that  $\mmod \CX$ is an abelian category closed under cokernels. Study of $\mmod\CX$ and its properties is known under the name of functor category. Towards the end of the paper, we shall use some known facts of the functor categories without further references. Just let us implement the following two facts.
\begin{itemize}
\item [(i)] Let $f: F_1 \rt F_2$ be a map in $\mmod\CX$  and  $(-, X_1) \rt (-, X_0) \rt F_1 \rt 0$ and $(-, Y_1)\rt (-, Y_0) \rt F_2 \rt 0$ be projective presentations of $F_1$ and $F_2$, respectively. In view of the proof of \cite[Proposition 2.1]{A}, we can construct the following projective  presentation of $\Coker  f$,  $$(-, X_0\oplus Y_1) \rt (-, Y_0) \rt\Coker f \rt 0.$$

\item [(ii)] Let $\CX$ be a subcategory of $\mmod \La$ containing $\prj\La$. One can identify the objects of the category
$\mmod \underline{\CX}$ of finitely presented functors over the stale category $\underline{\CX}$ with those functors of $\mmod \CX$ that vanish over  $\prj\La$.
\end{itemize}
\end{remark}

The proof of the following lemma is inspired by a part of the proof of  \cite[Proposition 2.1]{A}.

\begin{proposition}\label{PrjPresentation}
Let $\La$ be a Gorenstein algebra and $\CC$ be an $n\Z$-cluster tilting subcategory of $\mmod\La$.  Let
$$0\rt (-,\underline{A})\rt (-,\underline{M})\rt(-,\underline{M}'')$$
be an exact sequence in $\mmod\CC$, where for every object $B$, $(-,\underline{B})$  denotes  $\underline{\Hom}_{\La}(-,B)\vert _{\CC}$ and $M,M''\in\nGprj\CC$. Then $(-,\underline{A})$ has a projective presentation $(-,X_1)\rt (-,X_0)\rt (- ,\underline{A})\rt 0$ such that $X_0,X_1\in\nGprj\CC$.
\end{proposition}

\begin{proof}
Let  $M$ and $M''$ be objects of $\nGprj\CC$. Consider  the following exact sequences
$$0\rt\Omega^{n}(M)\rt P^{n-1}\rt\cdots\rt P^{0}\rt M\rt 0,  \ \ \ \ \ \ \ \ \ (\ast) $$
$$0\rt \Omega^{n}(M'')\rt Q^{n-1}\rt\cdots\rt Q^0\rt M''\rt 0. \ \ \ \ \ \ \ \ (\ast\ast)$$
  Clearly, $(-,P^0)\rt (-,M)\rt (-,\underline{M})\rt 0$ and $(-,Q^0)\rt(-,M'')\rt (-,\underline{M}'')\rt 0$ are projective presentations of $(-,\underline{M})$ and $(-,\underline{M}'')$ in $\mmod\CC$, respectively.
We know that map $(-,\underline{M})\rt (-,\underline{M}'')$ can be lifted to a map of the complexes
$\sigma: {\bf P} \rt \mathbf{Q}$, where ${\bf P}$ is the complex $$\cdots\rt 0 \rt (-, P^0) \rt (-, M) \rt 0\rt  \cdots,$$ and $\bf{Q}$  is the complex $$\cdots\rt 0\rt (-, Q^0) \rt (-, M'')\rt 0\rt  \cdots.$$ Let $\bf{M}(\sigma)$ be the mapping cone of $\sigma$, i.e.
 	$$\cdots\rt 0 \rt (-, P^0) \rt (-, M \oplus Q^0) \rt (-, M'')\rt 0 \rt \cdots $$

Hence we get an exact sequence
\[	\rm{H}^{-1}(\bf{P})\rt \rm{H}^{-1}(\bf{Q})\rt \rm{H}^{-1}(\bf{M}(\sigma))\rt\rm{H}^0(\bf{P})\rt \rm{H}^0(\bf{Q})\rt \rm{H}^0(\bf{M}(\sigma))\rt 0 	\]
with $\rm{H}^0({\bf P})=(-, \underline{M})$, $\rm{H}^0(\mathbf{Q})=(-, \underline{M}'')$ and the map $\rm{H}^0({\bf P}) \rt \rm{H}^0(\mathbf{Q})$ is the map $(-, \underline{M}) \rt (-, \underline{M}'')$.
Also by $(\ast)$ and $(\ast\ast)$, we get the following projective presentations
$$(-, P^2)\rt (-, P^1)\rt \rm{H}^{-1}(\mathbf{P})\rt 0,$$
$$(-, Q^2)\rt (-, Q^1)\rt \rm{H}^{-1}(\mathbf{Q})\rt 0,$$
for $\rm{H}^{-1}(\bf{P})$ and
$\rm{H}^{-1}(\bf{Q})$, respectively. In view of Remark \ref{ProjPresenCoker}, we have  the projective presentation
$(-, P^1\oplus Q^2)\rt (-, Q^1)\rt G \rt 0$ for $G$, where $G:=\Coker(\rm{H}^{-1}({\bf{P}}) \rt \rm{H}^{-1}({\bf{Q}}))$.  Now we show that there is a projective presentatin   $(-,C_1)\rt(-,C_0)\rt \rm{H^{-1}}({\bf{M}}(\sigma))\rt 0$ for $\rm{H^{-1}}(\bf{M}(\sigma))$, with $C_0,C_1\in\nGprj\CC$.  To do this, set  $Z:=\Ker((-, M\oplus Q^0)\rt (-, M''))$ in the mapping cone  $\bf{M}(\sigma)$. Since $M\oplus Q^0\rt M''$ is an epimorphism, by Lemma \ref{Condition(ii)}, there is an exact sequence
$$ 0  \rt A^n \rt \cdots A^1\rt M\oplus Q^0 \rt M'' \rt 0$$
such that $A^i \in \nGprj\CC.$ This sequence gives the following exact sequence
$$ 0  \rt(-, A^n) \rt \cdots (-, A^1)\rt (-, M\oplus Q^0) \rt  (-, M'') $$ in $\mmod \CC$. Consequently,  $Z$ has the projective presentation $(-, A^2)\rt (-, A^1) \rt Z\rt 0$ with $A^1, A^2 \in \nGprj\CC$. Therefore, by  Remark \ref{ProjPresenCoker}, $\rm{H}^{-1}({\bf{M}}(\sigma))=\Coker((-, P^0)\rt Z)$  has  the projective presentation $(-, P^0\oplus A^2) \rt (-, A^1)\rt \rm{H}^{-1}({\bf{M}}(\sigma))\rt 0$ in $\mmod \CC.$ \\
Now by considering the  short exact sequence $0\rt G\rt \rm{H}^{-1}({\bf{M}}(\sigma))\rt (-,\underline{A})\rt 0$, obtained from the  following diagram
\[ 	\xymatrix@C=0.4cm{
 		\rm{H}^{-1}({\bf{P}}) \ar[r] & \rm{H}^{-1}({\bf{Q}}) \ar[rr]^{} \ar@{->>}[dr] & & \rm{H}^{-1}({\bf{M}}(\sigma))\ar[r]^{}\ar@{->>}[d] &\rm{ H}^0({\bf{P}}) \ar[r] &\rm{ H}^0({\bf{Q}}) \\
 		& & G  \ar@{>->}[ur] & (-,\underline{A}) \ \ \ar@{>->}[ur] }
 	\]
 and another use of Remark \ref{ProjPresenCoker} again, we get the projective presentation $$(-, Q^0\oplus P^0 \oplus A^2)\rt (-, A^1)\rt (-, \underline{A})\rt 0$$
 for $(-,\underline{A})$ with $A^1, Q^0\oplus P^0\oplus A^2\in\nGprj\CC$. This completes the proof.
\end{proof}

Following fact will be used in the proof of the next theorem.
\begin{proposition}(see \cite{AR})\label{longexactStable}
Let $0 \rt X_2 \rt X_1 \rt X_0 \rt 0$ be an exact sequence in $\mmod \La.$ Then we have the  following long exact sequence
\[\begin{tikzcd}[column sep=tiny, row sep=tiny]
\cdots\rar&\underline{\rm{Hom}}_{\La}(-, \Om X_1)\rar&\underline{\Hom}_{\La}(-,\Om X_0)\rar&\underline{\Hom}_{\La}(-,X_2)\rar &{}\\     {}\rar&\underline{\Hom}_\La(-,X_1)\rar&\underline{\Hom}_\La(-,X_0)\rar&\Ext_{\La}^1(-,X_2)\rar&{}\\
{}\rar&\Ext_{\La}^1(,X_1)\rar&\Ext_{\La}^{1}(-,X_0)\rar&\Ext_{\La}^{2}(-,X_2)\rar&\cdots.
\end{tikzcd}\]
of functors on $\mmod \La$.

\end{proposition}
Now we are ready to state and prove our main theorem in this section.
\begin{theorem}\label{NresolvingGP}
Let $\La$ be a Gorenstein algebra and  $\CC$ be an $n\Z$-cluster tilting subcategory of $\mmod\La$. Then  $\nGprj\CC$ is an $n$-resolving subcategory of $\CC$.
\end{theorem}

\begin{proof}
By Proposition \ref{Prop 3.6}, $\nGprj\CC$ contains projectives, so condition $(i)$ of Definition \ref{Resolving} satisfies. Condition $(ii)$ follows directly from Lemma \ref{Condition(ii)}. To complete the proof, only the third condition of the definition should be investigated for $\nGprj\CC$.
Assume that exact sequence $$\epsilon: \ \ 0 \rt M' \rt A^1 \rt \cdots \rt A^{n} \rt M \rt 0,$$
with $M$ and $M'$ in $\nGprj\CC$ is given. Then, by \cite[Appendix of Chapeter VII]{M}, we observe that $\epsilon $ is Yoneda equivalent to $fE$, for some $f:\Om^n(M) \rt M',$  where $E$ denotes the following exact sequence
$$E: \ \ 0 \rt \Om^{n}(M)  \rt P^{n-1} \rt \cdots \rt P^1 \rt P^0 \rt M \rt 0.$$
In fact, $fE$ is just the following $n$-fold extension of $M'$ by $M$
$$fE: \ \ 0 \rt M' \rt U \rt P^{n-2} \rt \cdots \rt P^1 \rt P^0 \rt M \rt 0,$$
where $0 \rt M' \rt U \rt \Om^{n-1}(M)\rt 0$ is obtained as the  push-out of the short exact sequence $0 \rt \Om^n(M)\rt P^{n-1} \rt \Om^{n-1}(M)\rt 0$ along the morphism $f$.
The exact sequence $fE$ can be broken into the following short exact sequences
$$1 \leqslant i \leqslant n-1, \ \  \epsilon_i: \ 0 \rt \Om^i(M) \rt P^{i-1} \rt \Om^{i-1}(M) \rt 0,$$
$$ \ \ \  \  \ \  \  \ \ \  \ \ \  \ \ \ \ \ \  \ \  \ \ \delta:  \ 0 \rt M' \rt U \rt \Om^{n-1}(M) \rt 0,$$
where $\Om^0(M)=M$. Applying  the functor $\Hom_{\La}(C, -)$, for each $C \in \CC$, to the short exact sequences $\epsilon_i$, and looking at the corresponding long exact sequences, we can deduce the following isomorphisms of functors
$$   (-,\underline{M} ) \simeq \Ext^1_{\La}(-, \Om(M))\simeq \cdots \simeq \Ext^{n-1}_{\La}(-, \Om^{n-1}(M))   \ \ \ \ \ \ \ \ \ \  \  \ \ \ \dagger $$
in $\mmod \CC$.
On the other hand, the short exact sequence $(\delta)$ similarly induces the following exact sequence
$$ 0 \rt \Ext^{n-1}_{\La}(-, U) \rt \Ext^{n-1}_{\La}(-, \Om^{n-1}(M))\rt \Ext^n_{\La}(-, M') \ \ \ \ \ \ \ \ \ \   \dagger\dagger   $$
in $\mmod \CC.$ By combining $(\dagger)$ and $(\dagger \dagger)$, we obtain the following exact sequence in $\mmod \CC$
$$0 \rt \Ext_{\La}^{n-1}(-, U) \rt (-, \underline{M}) \st{\Phi} \rt \Ext^n_{\La}(-, M') \ \ \ \ \ \ \ \ \ \ \ \ \ \ \ \ \ \ \ \ \ \ \ \ \ \ \ \ \ \dagger\dagger\dagger $$
such that $\Phi_{\underline{M}}(\rm{Id}_{\underline{M}})=[fE].$ Let $F=\rm{Im} \Phi$.  According to \cite[A.1. Proposition]{I}, $fE$ and hence $\epsilon$ is Yoneda equivalent to an $n$-fold extension constructed by the minimal projective resolution of $F$  by adding, if necessary, the split exact complexes induced by $M$ and $M'$ to the minimal projective resolution. So to complete the proof of condition $(iii)$, we show that $F$ has a minimal projective resolution
 $$0\rt (-,C_{n+1})\rt\cdots\rt (-,C_0)\rt F\rt 0$$
 such that  for all $i\in\lbrace 0,...,n+1\rbrace$, $C_i\in\nGprj\CC$.

 Since $M\in \nGprj\CC$, $\Om^{n-1}(M) \in \Gprj\La$. This implies that  $\delta$ is $\Hom_{\La}(-, \prj\La)$-exact. Hence in  a straightforward way, the short exact sequence $\delta$ can be embedded to the following commutative diagram
$$\xymatrix@R=0.4cm@C=0.6cm{  & 0\ar[d] & 0\ar[d] &  & 0\ar[d] & 0\ar[d] & \\0 \ar[r] &
 		M'\ar[r]\ar[d] &
 		Q^{n-2}\ar[r]\ar[d] & \cdots\ar[r] &
 		Q^0\ar[r]\ar[d] & N\ar[d]\ar[r]& 0 \\0\ar[r] &
 		U\ar[r]\ar[d] & L^{n-2}\ar[r]\ar[d] & \cdots\ar[r] & L^0 \ar[r]^{f}\ar[d] & W\ar[d]\ar[r]& 0\\ 0\ar[r] &
 		\Om^{n-1}(M)\ar[r]\ar[d] & P^{n-2}\ar[r]\ar[d] & \cdots\ar[r] & P^0\ar[r]\ar[d] & M\ar[d]\ar[r] & 0\\  &
 		0 & 0 &  & 0 & 0 &  }$$
such that all the columns and rows are exact,  $L^i$ is projective, for all $i\in\lbrace 0,...,n-2\rbrace$, and also the  first row is obtained by the $n\Z$-totally acyclic complex  associated to $n\Z$-Gorenstein projective module $M'.$
In particular, by the above diagram we have $U \simeq \Om^{n-1}(W)$ in $\underline{\rm{mod}}\mbox{-}\La.$ Since $\CC \subseteq {}^{\perp_{n-1}}\prj\La$, by \cite[Subsection 2.1]{I1}, we have the following series of natural  isomorphisms of functors on $\CC$
\begin{align*}
\Ext^{n-1}_{\La}(-, U)&= \underline{\rm{Hom}}_{\La}(\Om^{n-1}(-),  U)\mid_{\CC}  \\&
\simeq \underline{\rm{Hom}}_{\La}(\Om^{n-1}(-),  \Om^{n-1}(W))\mid_{\CC}  \\
& \simeq \underline{\rm{Hom}}_{\La}(-, W)\mid_{\CC}.
\end{align*}
So by  $(\dagger\dagger\dagger)$, we have  the following short exact sequence
$$0\rt (-,\underline{W})\rt (-,\underline{M})\rt F\rt 0. $$
From now on, for simplicity, we write $(-,\underline{A})$ for  $\underline{\rm{Hom}}_{\La}(-, A)\mid_{\CC}$, and also delete $\mid_{\CC}$ when we consider  functors on $\CC.$

 We claim that $(-,\underline{W})$ is the kernel of the map $(-,\underline{M})\rt (-,\underline{M}'')$ with $M,M''\in\nGprj\CC$. To do this, consider the short exact sequence $0 \rt N \rt W \rt M \rt 0$, located at the right most column in the above diagram. By Proposition \ref{longexactStable},
 we get  the following long exact sequence of functors on $\mmod \La$
 	\[\begin{tikzcd}[column sep=tiny, row sep=tiny]
\cdots\rar&\underline{\rm{Hom}}_{\La}(-, \Om M)\rar&\underline{\Hom}_{\La}(-, N)\rar&\underline{\Hom}_{\La}(-,W)\rar &{}\\  {}\rar&\underline{\Hom}_\La(-,M)\rar&\Ext_{\La}^1(-,N)\rar&
\Ext_{\La}^1(,W)\rar& \Ext_{\La}^1(-,M)\rar&\cdots.
\end{tikzcd}\]	
 The restriction of the above diagram to the objects of $\CC$ gives us the following long exact sequence
 \[(-, \underline{\Om(M)})\rt(-,\underline{N})\rt(-,\underline{W})\rt(-, \underline{M})\rt\Ext^1_{\La}(-, N)\rt\Ext_{\La}^1(-, W)\rt0 \ \ \ddagger \]
 	in $\mmod \CC.$
 Since  $M' \in \nGprj\CC$, there is a short exact sequence $0 \rt N \rt Q \rt M''\rt 0$, in which $M'' \in \nGprj\CC$ and $Q \in \prj\La.$ This short exact sequence induces  the following exact sequence 	
 \[	0\rt(-,N)\rt(-,Q)\rt(-, M'')\rt\Ext^1_{\La}(-, N)\rt 0.	
 	\]
 So one can deduce  the isomorphism $\Ext^{1}_{\La}(-, N)\simeq (-, \underline{M}'')$ in $\mmod \CC$ and also $(-,\underline{N})=0$.
 Thus, by considering above facts, the long exact sequence $(\ddagger)$ can be written as
 	\[
 	0\rt (-,\underline{W})\rt (-, \underline{M})\rt (-, \underline{M}'')\rt \Ext_{\La}^1(-, W) \rt 0.
 	\]
In particular, $(-, \underline{W})$ is the kernel of the map $(-, \underline{M}) \rt (-, \underline{M}'')$ with $M, M'' \in \nGprj\CC,$ as it was claimed. By Lemma \ref{PrjPresentation}, there exists  a projective presentation $(-,X_1)\rt (-,X_0)\rt (-, \underline{W})\rt   0$ with $X_1,X_0\in\nGprj\CC$. So, by Remark \ref{ProjPresenCoker} and short exact sequence
$0\rt  (-, \underline{W})\rt (-, \underline{M}) \rt F \rt 0,$  we deduce that $F$ has a projective presentation $(-,C_1)\rt(-,C_0)\rt F\rt 0$ such that $C_0,C_1\in\nGprj\CC$.
Since $F(\La)=0$, the induced map $C_1\rt C^0$ is  an epimorphism. Now by Lemma \ref{Condition(ii)} for the epimorphism $C_1 \rt C_0$, we get an exact sequence
$$0 \rt C_{n+1} \rt \cdots C_2 \rt C_1 \rt C_0 \rt 0 $$
 with $C_i\in\nGprj\CC$, for all $i\in\lbrace 0,...,n+1\rbrace$. This, in turn,  implies the following projective resolution of $F$ in $\mmod \CC$
$$0 \rt (-, C_{n+1}) \rt \cdots (-, C_2) \rt (-, C_1) \rt (-, C_0) \rt F \rt 0 $$	
Consequently, the minimal projective resolution of $F$ in $\mmod \CC$ consists of only the objects in $\nGprj\CC.$ The proof is hence complete.
\end{proof}

\section{Observation}
Let $\La$ be an artin algebra. The quotient category $\D^b(\La) /\D^{\text{perf}}(\La)$,  where $\D^{b}(\La)$ denotes the bounded derived category of finitely generated $\La$-modules and  $\D^{\text{perf}}(\La)$ denotes the full triangulated subcategory of $\D^{b}(\La)$ consisting of objects isomorphic to bounded complexes of finitely generated projective $\La$-moules, is called the singularity category of $\La$, denoted by $\D_{\rm{sg}}(\La)$.

Buchweitz \cite{Bu}, Happel \cite{Hap} and Rickard \cite{Ri} independently have shown that $\D_{\rm{sg}}(\La)$ is triangule equivalence to $\underline{\Gprj}\La$, if $\La$ is a Gorenstein algebra. In the following, we provide an observation for  higher version of this known equivalence.

 Let $\La$ be a Gorenstein algebra,  $\Phi: \underline{\Gprj}\La \rt \D_{\rm{sg}}(\La)$ be the above  mentioned  equivalent  of triangulated categories and $\CC$ be  an $n\Z$-cluster tilting subcategory of $\mmod\La$.
 One can see that $\Phi$ maps $\nGprj\CC$ to the subcategory
 \[\Upsilon=\{P^{\bullet} \in\D_{\rm{sg}}(\La)\mid P^{\bullet}\simeq C[ni] \ \text{for some} \ C \in \CC \ \text{and} \ i \in \Z \},\]
 of $\D_{\rm{sg}}(\La)$.
Therefore, we have the following commutative diagram
\[ \xymatrix{  \underline{\Gprj}\La \ar[rr]^<<<<<<<<<<{\Phi} && \D_{\rm{sg}}(\La) \\
	\underline{\nGprj}\CC \ar@{^(->}[u] \ar[rr]^<<<<<<<<<{\Phi\mid} && \Upsilon  \ar@{^(->}[u] }\]
 with  equivalences in the rows. Since by Theorem \ref{NClusterTilting}, $\underline{\nGprj}\CC$ is an $n\Z$-cluster tilting subcategory  of $\underline{\Gprj}\La$,  the above diagram implies that  $\Upsilon$
 is $n\Z$-cluster tilting subcategory of $\D_{\rm{sg}}(\La)$ and hence by \cite[Theorem 1 ]{GKO}, $\Upsilon$ gets $(n+2)$-angulated structure. Moreover, since  by Corollary \ref{Nangulatedcategory}, $\underline{\nGprj}\CC$ is $(n+2)$-angulated category,  restriction functor $\Phi\mid$ is an equivalence of $(n+2)$-angulated categories. We point out that this  observation should be compared with a result of Kevamme \cite{K} and as mentioned there $\Upsilon$ can be considered as  a higher analogue of the singularity category.

 \section*{Acknowledgments}
The authors would like to express their special thanks to Professor Apostolos Beligiannis for suggesting this project and for several discussions he have had with the third author, during her visit at the University of Ioannina. Part of this work is done during the first author's visit at the University of Leicester. He would like to thank Professor Sibylle Schroll for her support, warm hospitality and excellent mathematical discussions.

\end{document}